\documentclass{article}

\usepackage{amsmath}
\usepackage{amssymb}
\usepackage{amsthm}
\usepackage[all]{xy}

\newtheorem{num}{Numbering}[section]
\newtheorem{theorem}[num]{Theorem}

\newtheorem{lemma}[num]{Lemma}
\newtheorem{corollary}[num]{Corollary}

\newtheorem{proposition}[num]{Proposition}

\title{Extensions of Scott's Graph Model and Kleene's Second Algebra}
\author{Jaap van Oosten\footnote{Corresponding author. Department of Mathematics, Utrecht University, {\tt j.vanoosten@uu.nl}} \and Niels Voorneveld\footnote{Department of Mathematics, University of Ljubljana, {\tt niels.voorneveld@fmf.uni-lj.si}}}
\date{October 13, 2016}

\begin{document}
\maketitle

\begin{abstract}
	We use a way to extend partial combinatory algebras (pcas) by forcing them to represent certain functions. In the case of Scott's Graph model, equality is computable relative to the complement function. However, the converse is not true. This creates a hierarchy of pcas which relates to similar structures of extensions on other pcas. We study one such structure on Kleene's second model and one on a pca equivalent but not isomorphic to it. For the recursively enumerable sub pca of the Graph model, results differ as we can compute the (partial) complement function using the equality.
\end{abstract}

\section*{Introduction}
In this paper we study extensions of various partial combinatory algebras; mainly partial combinatory algebra structures on the power set of the natural numbers and the set of all functions from natural numbers to natural numbers.

Partial combinatory algebras have been useful in the past for devising realizability interpretations of intuitionistic formal systems. However, there is another side to them: they can be viewed as paradigms of computation. It is this view that has been put forward in several publications of the first author, but it has its origin in the seminal thesis \cite{Long} of John Longley. Longley defined a notion of morphism between partial combinatory algebras which, whilst fundamental in the study of realizability toposes, also has a clear computational meaning: a morphism $\mathcal{A} \to \mathcal{B}$ is a way to simulate the computations of $\mathcal{A}$ in $\mathcal{B}$.

In \cite{vOrelrec}, the first author showed how, given a partial combinatory algebra $\mathcal{A}$ and an arbitrary partial endofunction $f$ on $\mathcal{A}$, one can construct, in a universal way, a partial combinatory algebra $\mathcal{A}[f]$ in which $f$ is ``computable''; the construction is a straightforward generalization of Turing's notion of ``oracle''. The construction gives a clear meaning to statements like ``$f$ is computable in $g$ relative to the partial combinatory algebra $\mathcal{A}$''.

This paper is structured as follows. After a section on preliminaries which contains all basic definitions, in section 2 we introduce the various extensions we are interested in. We look at ${\cal P}(\mathbb{N})[C]$, where ${\cal P}(\mathbb{N})$ is Scott's graph model and $C:{\cal P}(\mathbb{N})\to {\cal P}(\mathbb{N})$ is the complement function. We prove that ${\cal P}(\mathbb{N})[C]$ is decidable. Then we look at related extensions of ${\cal K}_2$ (Kleene's pca on the set of functions $\mathbb{N} \to \mathbb{N}$) and a pca (called $2^{\omega}$) on the set of functions $\mathbb{N} \to \{ 0,1\}$. We characterize the topologies on these pcas which are of interest (interacting nicely with the applicative structure). We prove that the pcas ${\cal K}_2$ and $2^{\omega}$ are equivalent, but not isomorphic (to our knowledge, the first example of this phenomenon in the literature).

In section 3, Independence results, we prove that (conversely to the result in the previous section), the complement function is {\em not\/} computable, relative to ${\cal P}(\mathbb{N})$, in a decisioin function for equality; and we have the analogous results for the corresponding extensions of ${\cal K}_2$ and $2^{\omega}$. The methods used also yield a non-existence result for decidable applicative morphisms into ${\cal P}(\mathbb{N})$, ${\cal K}_2$ and $2^{\omega}$.

In a final section we discuss recursive or r.e.\ sub-partial combinatory algebras. Also here we have `complement-like' functions which however, now are strictly partial (as in the RE submodel of ${\cal P}(\mathbb{N})$). The results obtained are different: the partial complement function on RE is computable in the equality relation.

We believe that our work is a contribution towards the Higher-Order Computability programme of Longley and Normann (\cite{LongNor}).

This paper originates in the second author's master dissertation \cite{Voorneveld}. The first author acknowledges with gratitude the hospitality of the mathematics department of the University of Ljubljana, where he spent a sabbatical stay in the fall of 2016.

\section{Preliminaries}\label{prelimsection}
A {\em partial applicative structure\/} (pas) is a set $\mathcal{A}$ together with a partial {\em application\/} function $\mathcal{A}\times \mathcal{A}\to \mathcal{A}$, which we write $a,b\mapsto ab$. We write $ab{\downarrow}$ to mean that the pair $(a,b)$ is in the domain of the application function. If we have a more complicated term $s$, we write $s{\downarrow}$ to mean that, not only, $s$ denotes but also all subterms of $s$ do.

A {\em partial combinatory algebra\/} (pca)  is a pas satisfying the following axiom: there are elements ${\sf k}$ and $\sf s$ in $\mathcal{A}$ such that for all $a,b,c\in \mathcal{A}$:\begin{itemize}
	\item[] $({\sf k}a)b=a$ (in particular, ${\sf k}a{\downarrow}$).
	\item[] $({\sf s}a)b{\downarrow}$, and, whenever $(ac)(bc){\downarrow}$, $(({\sf s}a)b)c{\downarrow}$ and $(({\sf s}a)b)c=(ac)(bc)$.\end{itemize}
From now on, we economise on brackets, and {\em associate to the left}: we write, e.g., $sabc$ for $((sa)b)c$ and $aa_1\cdots a_n$ for $(\cdots ((aa_1)a_2)\cdots )a_n$. In a nontrivial pca, the application is never associative so $abc$ is in general different from $a(bc)$.

A term $t$ composed of elements of $\mathcal{A}$, variables and the application function, represents a partial function $\mathcal{A}^n\to \mathcal{A}$
(where $n$ is the number of variables in $t$), and again we write $t(a_1,\ldots ,a_n){\downarrow}$ to mean that the tuple $(a_1,\ldots ,a_n)$ is in the domain of the function.

Now suppose the term $t$ has variables $x_1,\ldots ,x_{n+1}$. There is an element ${<} x_1\cdots x_{n+1}{>}t$ with the following property: for every $n+1$-tuple $a_1,\ldots ,a_{n+1}$ from $\mathcal{A}$ we have\begin{itemize}
	\item[] $({<}x_1\cdots x_{n+1}{>}t)a_1\cdots a_{n}{\downarrow}$
	\item[] If $t(a_1,\ldots ,a_{n+1}){\downarrow}$ then $({<}x_1\cdots x_{n+1}{>}t)a_1\cdots a_{n+1}{\downarrow}$, and both have the same value in $\mathcal{A}$.\end{itemize}
For example, we might use ${<}xyz{>}xz(yz)$ for the element ${\sf s}$ of $\mathcal{A}$.

A pca contains {\em booleans\/} $\sf T$ and $\sf F$, and a {\em definition by cases\/} operator $C\in \mathcal{A}$, satisfying for all $a,b\in \mathcal{A}$:\begin{itemize}
	\item[] $C{\sf T}ab = a$ and $C{\sf F}ab=b$\end{itemize}
We usually refer to the term $Cxab$ as\begin{itemize}
	\item[] {\sf If} $x$ {\sf then} $a$ {\sf else} $b$\end{itemize}
Note that this gives us the boolean operations `and' and `not':\begin{itemize}
	\item[] {\sf and} $ab$  = {\sf If} $a$ {\sf then} ({\sf if} $b$ {\sf then T else F}) {\sf else F}
	\item[] {\sf not} $a$ = {\sf If} $a$ {\sf then F else T}\end{itemize}
In every pca, one can code pairs and sequences; we shall use the notations $[a,b]$ and $[a_0,\ldots ,a_{n-1}]$ for codings of the pair $(a,b)$ and the sequence $(a_0,\ldots ,a_{n-1})$ respectively.

Let $f:\mathcal{A}^n\to \mathcal{A}$ be a partial function. We say that $f$ is {\em representable in $\mathcal{A}$} if there is some $a_f\in \mathcal{A}$ such that for every $n$-tuple $a_1,\ldots ,a_n$ in the domain of $f$, we have $a_fa_1\cdots a_n=f(a_1,\ldots ,a_n)$.

For a pca $\mathcal{A}$ and a subset $R\subset \mathcal{A}^n$, we call the set $R$ {\em decidable in\/} $\mathcal{A}$ if the function which sends each coded $n$-tuple $[a_1,\ldots ,a_n]$ to {\sf T} if $(a_1,\ldots ,a_n)\in R$, and to {\sf F} otherwise, is representable in $\mathcal{A}$. The pca $\mathcal{A}$ is called {\em decidable\/} if the equality relation is decidable in $\mathcal{A}$.
\subsection{Topology}\label{topologysection}
In this paper we shall be interested in pca structures on sets as ${\cal P}(\mathbb{N})$ or $\mathbb{N}^\mathbb{N}$, which have several interesting topologies. We introduce the following terminology for studying the interaction of these topologies with the pca structure.

A topology on a pca $\mathcal{A}$ is {\em repcon\/} if every partial representable function is continuous on its domain (as subspace of $\mathcal{A}$); a topology is {\em conrep\/} if every partial continuous function is representable.

Clearly, the discrete and indiscrete topologies are always repcon. We refer to these topologies as the trivial topologies.
\subsection{Applicative morphisms}\label{appmorsection}
Let $\mathcal{A}$ and $\mathcal{B}$ be pcas and $\gamma$ a total relation from $\mathcal{A}$ to $\mathcal{B}$. That is, $\gamma$ assigns to every $a\in \mathcal{A}$ a nonempty subset $\gamma (a)$ of $\mathcal{B}$.

A partial function $f:\mathcal{A}^n\to \mathcal{A}$ is said to be {\em representable in $\mathcal{B}$ with respect to $\gamma$}, if there is an element $b_f\in \mathcal{B}$ such that for any $n$-tuple $(a_1,\ldots ,a_n)$ in the domain of $f$ we have: whenever $b_1\in\gamma (a_1),\ldots ,b_n\in\gamma (a_n)$ then $b_fb_1\cdots b_n{\downarrow}$ and is an element of $\gamma (f(a_1,\ldots ,a_n))$.
Such a total relation $\gamma$ is called an {\em applicative morphism\/} from $\mathcal{A}$ to $\mathcal{B}$ if the application function of $\mathcal{A}$ is representable in $\mathcal{B}$ w.r.t.\ $\gamma$: so there should be an element $r\in \mathcal{B}$ (the {\em realizer\/} of the applicative morphism $\gamma$) such that whenever $aa'{\downarrow}$ in $\mathcal{A}$ and $b\in \gamma (a),b'\in\gamma (a')$, we have $rbb'\in
\gamma (aa')$. When we view pcas as models of computation, an applicative morphism can be seen as a way to simulate $\mathcal{A}$-computations in $\mathcal{B}$.

Given two applicative morphisms $\gamma ,\delta :\mathcal{A} \to \mathcal{B}$ we say $\gamma\leq\delta$ if there is some $s\in \mathcal{B}$ such that for all $a\in \mathcal{A}$ and $b\in\gamma (a)$, $sb{\downarrow}$ and $sb\in\delta (a)$.

With composition of relations (which preserves the preorder $\leq$), and the identity relations, we have a preorder-enriched category {\bf PCA} of pcas, applicative morphisms and inequalities.

We single out a subcategory of {\bf PCA}. An applicative morphism $\gamma :\mathcal{A} \to \mathcal{B}$ is {\em decidable\/} if there is an element $d\in \mathcal{B}$ which satisfies the following: if ${\sf T}_{\mathcal{A}}$, ${\sf F}_{\mathcal{A}}$ denote the booleans in $\mathcal{A}$, then for $b\in\gamma ({\sf T}_{\mathcal{A}}), c\in\gamma ({\sf F}_{\mathcal{A}})$ we have $db={\sf T}_{\mathcal{B}}$, $dc={\sf F}_{\mathcal{B}}$ (where of course ${\sf T}_{\mathcal{B}}$ and ${\sf F}_{\mathcal{B}}$ are the booleans in $\mathcal{B}$).

The theory of applicative morphisms is due to John Longley (\cite{Long}).
\subsection{Extensions of pcas by functions}\label{extensionsection}
In \cite{vOrelrec} the following theorem is proved:
\begin{theorem}\label{extensiontheorem} For any pca $\mathcal{A}$ and partial endofunction $f$ on $\mathcal{A}$, there is a pca $\mathcal{A}[f]$ with the same underlying set $\mathcal{A}$, and a decidable applicative morphism $\iota _f:\mathcal{A} \to \mathcal{A}[f]$, which is the identity relation on $\mathcal{A}$, such that the function $f$ is representable in $\mathcal{A}[f]$ with respect to $\iota _f$, and moreover, for any decidable applicative morphism $\gamma :\mathcal{A} \to \mathcal{B}$ such that $f$ is representable in $\mathcal{B}$ with respect to $\gamma$, there is a unique factorisation of $\gamma$ as $\gamma =\gamma _f\iota _f$, for a decidable morphism $\gamma _f:\mathcal{A}[f]\to \mathcal{B}$.\end{theorem} 
In this paper we shall need a detail in the construction of $\mathcal{A}[f]$. The application in $\mathcal{A}[f]$, written $a,b\mapsto a{\cdot}_fb$, is defined as follows:
\begin{itemize}
	\item[] $a{\cdot}_fb=c$ iff there is a sequence $e_0,\ldots ,e_{n-1}$ of elements of $\mathcal{A}$ such that for all $i<n$,
	$$a[b,f(e_0),\ldots ,f(e_{i-1})]=[{\sf F},e_i]$$
	and $a[b,f(e_0),\ldots ,f(e_{n-1})]=[{\sf T},c]$.
\end{itemize}
\section{Variations on Scott's Graph Model and \\ Kleene's Second Model}\label{variationssection}
Let us agree on the following conventions for the set $\mathbb{N}$  of natural numbers: $0\in \mathbb{N}$; by $[{\cdot},{\cdot}]$ and $[{\cdot},\ldots ,{\cdot}]$ we denote standard pairing and sequence coding functions; we employ the following coding of finite subsets of $\mathbb{N}$: $p=e_k$ (the natural number $k$ codes the finite set $p$) if $k=\sum_{i\in p}2^i$ (so, $e_0=\emptyset$, $e_{2^n}=\{ n\}$ etc.).

Scott's Graph Model is the pca structure on ${\cal P}(\mathbb{N})$ given by the following application:
$$A{\circ}B\; =\; \{ n\, |\, \exists m([m,n]\in A, e_m\subseteq B)\}$$
It is well-known that the functions ${\cal P}(\mathbb{N})\to {\cal P}(\mathbb{N})$ which are representable in the pca ${\cal P}(\mathbb{N})$ are precisely the Scott-continuous functions. The {\em Scott topology\/} on ${\cal P}(\mathbb{N})=\{ 0,1\} ^\mathbb{N}$ is the product topology on the countable product of copies of $\{ 0,1\}$, where $\{ 0,1\}$ has the Sierpinski topology (with open sets $\emptyset ,\{ 1\} , \{ 0,1\}$). Concretely, a subset $X$ of ${\cal P}(\mathbb{N})$ is open if for any $A\in X$ there is a finite subset $p$ of $A$ such that every superset of $p$ is in $X$. We use the notation
$$U_p\; =\;\{ A\subseteq \mathbb{N}\, |\, p\subseteq A\}$$
(for finite $p$) for basis elements in this topology.

\noindent In the terminology of section~\ref{topologysection}, the Scott topology is both conrep and repcon for Scott's graph model.

We have the following result about minimal repcon topologies on ${\cal P}(\mathbb{N})$:
\begin{lemma}\label{minimalrepconPw}Let $\cal T$ be a nontrivial repcon topology for ${\cal P}(\mathbb{N})$. Then $\cal T$ either contains the Scott topology, or the Alexandroff topology (which contains precisely the downward closed sets w.r.t.\ $\subseteq$).\end{lemma}
\begin{proof} Since $\cal T$ is nontrivial, it contains an open set $U$ which is neither $\emptyset$ nor ${\cal P}(\mathbb{N})$. We distinguisg two cases:
	\medskip
	
	\noindent {\sc case} 1: $\emptyset \notin U$. Take $C\in U$, and let $p=e_n$ be an arbitrary finite subset of $\mathbb{N}$. Let $A=\{ [n,m]\, |\, m\in C\}$. The map $A{\circ}(-)$ is continuous for $\cal T$, so the set $\{B\, |\, A{\circ}B\in U\}$ is in $\cal T$. Now $A{\circ}B=C$ if $p\subseteq B$ and $\emptyset$ else; because $\emptyset\not\in U$ by assumption, the set $U_p$ is in $\cal T$. Since $p$ was arbitrary, $\cal T$ contains the Scott topology.
	\medskip
	
	\noindent {\sc case} 2: $\emptyset\in U$. Take $C\not\in U$; let $S\subseteq \mathbb{N}$ arbitrary. Let $A=\{ [2^n,m]\, |\, n\not\in S, m\in C\}$. Again, the map $A{\circ}(-)$ is $\cal T$-continuous, so $\{ B\, |\, A{\circ}B\in U\}$ is in $\cal T$. Now $A{\circ}B=\emptyset$ if $B\subseteq S$ and $C$ otherwise; by assumption on $C$ and $U$ we see that $\{ B\, |\, A{\circ}B\in U\} =\{ B\, |\, B\subseteq S\}$. So $\cal T$ contains the Alexandroff topology.
\end{proof}

\noindent A function which is definitely {\em not\/} Scott-continuous is the complement function $C(A)= \mathbb{N}-A$ on ${\cal P}(\mathbb{N})$. We shall study ${\cal P}(\mathbb{N})[C]$, the pca formed as in section~\ref{extensionsection}. According to the theorem quoted there, we have a decidable applicative morphism $\iota _C:{\cal P}(\mathbb{N})\to {\cal P}(\mathbb{N})[C]$ with the stated universal property.

We have the following corollary of lemma~\ref{minimalrepconPw}, about repcon topologies for ${\cal P}(\mathbb{N})[C]$:
\begin{corollary}\label{PCnorepcon}The only repcon topologies for ${\cal P}(\mathbb{N})[C]$ are the trivial ones.\end{corollary}
\begin{proof} 
	Suppose $\cal T$ is a non-trivial repcon for ${\cal P}(\mathbb{N})[C]$. Since every representable map in ${\cal P}(\mathbb{N})$ is also representable in ${\cal P}(\mathbb{N})[C]$, by lemma~\ref{minimalrepconPw} $\cal T$ contains either the Scott topology or the Alexandroff topology. In both cases, there is a nontrivial open set $U$ with either $\emptyset\in U$ or $\mathbb{N}\in U$. But the complement function is representable too, so $\{ \mathbb{N}-A\, |\, A\in U\}$ is also open; so we always have a nontrivial open $U$ with $\emptyset\in U$. By the proof of \ref{minimalrepconPw}, for any set $S$, we have that $\{ B\, |\, B\subseteq S\}$ is open. Again applying the continuity of the complement function we find that $\{ B\, |\, (\mathbb{N}-B)\subseteq (\mathbb{N}-S)\}$ is also open. Hence their intersection is open, but this is $\{ S\}$; so ${\cal T}$ is discrete. We have a contradiction.
\end{proof}

\noindent We denote application in ${\cal P}(\mathbb{N})[C]$ by $A,B\mapsto A{\cdot}B$. There are elements $\sf r$ and $\sf c$ in ${\cal P}(\mathbb{N})$ such that $\sf r$ realizes the applicative morphism $\iota _C$ and $\sf c$ represents the function $C$:
$$\begin{array}{rcl} {\sf r}{\cdot}A{\cdot}B & = & A{\circ}B \\ {\sf c}{\cdot}A & = & \mathbb{N}-A\end{array}$$
\begin{lemma}\label{P[C]Booleans} Let $\sf T$ and $\sf F$ be the booleans in ${\cal P}(\mathbb{N})[C]$. There is an element $\sf n$ of ${\cal P}(\mathbb{N})$ satisfying ${\sf n}{\cdot}\emptyset ={\sf F}$ and ${\sf n}{\cdot}\{ 0\} ={\sf T}$. Hence, we can take $\emptyset$ and $\{ 0\}$ for the Booleans in ${\cal P}(\mathbb{N})[C]$.\end{lemma}
\begin{proof} 
	Let $\ast$ be the binary operation given by $A{\ast}B=A{\circ}(\mathbb{N}-B)$. Note that in ${\cal P}(\mathbb{N})[C]$ this operation is represented by
	$${\sf s}={<}xy{>}{\sf r}{\cdot}x{\cdot}({\sf c}{\cdot}y)$$
	Let $M$ be the set $\{ [1,[2,x]]\, |\, x\in {\sf T}\}\cup \{ [0,[1,x]]\, |\, x\in {\sf F}\}$.
	
	Then $M{\circ}\emptyset =\{ [1,x]\, |\, x\in {\sf F}\}$ and 
	$$M{\circ}\{ 0\} =\{ [2,x]\, |\, x\in {\sf T}\}\cup\{ [1,x]\, |\, x\in {\sf F}\}$$
	Define ${\sf n}\, =\, {<}x{>}{\sf r}{\cdot}({\sf r}{\cdot}M{\cdot}x){\cdot}({\sf c}{\cdot}x)$
	Then
	$$\begin{array}{rclcl} {\sf n}{\cdot}A & = & {\sf r}{\cdot}({\sf r}{\cdot}M{\cdot}A){\cdot}({\sf c}{\cdot}A) & = & {\sf r}{\cdot}(M{\circ}A){\cdot}(\mathbb{N}-A) \\ & = & (M{\circ}A){\circ}(\mathbb{N}-A) & = & (M\circ A){\ast}A\end{array}$$
	hence
	$$\begin{array}{rclcl} {\sf n}{\cdot}\emptyset & = & (M{\circ}\emptyset ){\ast}\emptyset & = & \{ [1,x]\, |\, x\in {\sf F}\}{\ast}\emptyset \\ & = & \{ [1,x]\, |\, x\in {\sf F}\}{\circ}\mathbb{N} & = & {\sf F} \end{array}$$
	and
	$$\begin{array}{rcl} {\sf n}{\cdot}\{ 0\} & = & (M{\circ}\{ 0\} ){\circ} (\mathbb{N}-\{ 0\} ) \\
	& = & (\{ [2,x]\, |\, x\in {\sf T}\}\cup\{ [1,x]\, |\, x\in {\sf F}\} ){\circ}(\mathbb{N}-\{ 0\} ) \\
	& = & {\sf T}\end{array}$$
\end{proof}

\begin{theorem}\label{Pw[C]decidable} The pca ${\cal P}(\mathbb{N})[C]$ is decidable.\end{theorem}

\begin{proof} 
	In view of Lemma~\ref{P[C]Booleans} we need to exhibit an element $X$ of ${\cal P}(\mathbb{N})$ such that, in ${\cal P}(\mathbb{N})[C]$,
	$$X{\cdot}A{\cdot}B\; =\;\left\{\begin{array}{ll} \emptyset & \text{if } A\neq B\\ \{ 0\} & \text{otherwise}\end{array}\right.$$
	Since the pair $(\emptyset ,\{ 0\} )$ is a good pair of Booleans in ${\cal P}(\mathbb{N})[C]$, we have elements {\sf and} and {\sf not} in ${\cal P}(\mathbb{N})$ representing the indicated boolean operations for these booleans. Since in ${\cal P}(\mathbb{N})$, the set $P=\{ [2^x,[2^x,0]]\, |\, x\in \mathbb{N}\}$ has the property that $P{\circ}A{\circ}B=\{ 0\, |\, A\cap B\neq\emptyset\}$, we also have such an element in ${\cal P}(\mathbb{N})[C]$: an $S$ such that
	$$S{\cdot}A{\cdot}B\, =\,\left\{\begin{array}{ll}\emptyset & \text{if } A\cap B=\emptyset \\ \{ 0\} & \text{otherwise}\end{array}\right.$$
	Now define in ${\cal P}(\mathbb{N})[C]$ the element
	$$M\; =\; {<}yx{>}({\sf not} (S{\cdot}({\sf c}{\cdot}y){\cdot}x)){\sf and}({\sf not}(S{\cdot}y{\cdot}({\sf c}{\cdot} x)))$$
	where the element $\sf c$, again, represents the complement function.
	
	Then $M{\cdot}A{\cdot}B=\{ 0\}$ iff ${\sf not}(S{\cdot}(\mathbb{N}-A){\cdot}B)=\{ 0\}$ and ${\sf not}(S{\cdot}A{\cdot}(\mathbb{N}-B))=\{ 0\}$. That is, iff $B\cap (\mathbb{N}-A)=\emptyset$ and $(\mathbb{N}-B)\cap A=\emptyset$; i.e. iff $A=B$. Moreover. $M{\cdot}A{\cdot}B=\emptyset$ otherwise.
\end{proof}

We also consider the {\em Cantor\/} topology on ${\cal P}(\mathbb{N})$, the subspace topology of ${\cal P}(\mathbb{N})$ when embedded in the real line as the Cantor set; a basis for this topology is given by the sets
$$U^q_p\; =\; \{ A\subseteq \mathbb{N}\, |\, p\subseteq A, A\cap q=\emptyset\}$$
for disjoint, finite $p,q$.

\begin{lemma}\label{CantorconrepPw[C]}
	The Cantor topology is a conrep topology for ${\cal P}(\mathbb{N})[C]$.
\end{lemma}

\begin{proof} 
	Let $F$ be partial continuous on ${\cal P}(\mathbb{N})$ for the Cantor topology. Let {\sf r} and {\sf s} be as in the proof of \ref{P[C]Booleans}; so ${\sf r}{\cdot}A{\cdot}B=A{\circ}B$ and ${\sf s}{\cdot}A{\cdot}B=A{\circ}(\mathbb{N}-B)$.
	
	For every $n\in \mathbb{N}$ the set $V^n=\{ A\in {\rm dom}(F)\, |\, n\in F(A)\}$ is open in ${\rm dom}(F)$ by assumption, hence equal to the intersection of ${\rm dom}(F)$ with a set of the form
	$$\bigcup_{i\in I_n}U_{p(i,n)}^{q(i,n)}$$
	Consider the set
	$$W\, =\,\{ [a,[b,k]]\, |\, k\in \mathbb{N}, \exists i,j\in I_k(p(i,k)=e_a,q(i,k)=e_b)\}$$
	and let $Z\; =\; {<}x{>}{\sf s}{\cdot}({\sf r}{\cdot}W{\cdot}x){\cdot}x$. So $Z{\cdot}A=(W{\circ}A){\circ}(\mathbb{N}-A)$.
	
	We claim that $Z$ represents $F$ in ${\cal P}(\mathbb{N})[C]$. Suppose $A\in {\rm dom}(F)$. Now
	$$\begin{array}{rcl}W{\circ}A & = & \{ m\, |\, \exists n(e_n\subseteq A,[n,m]\in W)\} \\
	& = & \{ [b,k]\, |\,\exists a\exists i,j\in I_k(e_a\subseteq A,p(i,k)=e_a,q(i,k)=e_b)\} \\
	& = & \{ [b,k]\, |\,\exists i,j\in I_k(p(i,k)\subseteq A,q(i,k)=e_b)\} \end{array}$$
	hence
	$$\begin{array}{rcl} (W{\circ}A){\circ}(\mathbb{N}-A) & = & \{ m\, |\, \exists n(A\cap e_n=\emptyset ,[n,m]\in W{\circ}A)\} \\
	& = & \{ k\, |\,\exists b\exists i,j\in I_k(A\cap e_b=\emptyset ,p(i,k)\subseteq A,q(i,k)=e_b)\} \\
	& = & \{ k\, |\,\exists i,j\in I_k(p(i,k)\subseteq A, A\cap q(i,k)=\emptyset )\} \\
	& = & \{ k\, |\, A\in V^k\} \\ & = & F(A)\end{array}$$
\end{proof}

\noindent We wish to compare the pca ${\cal P}(\mathbb{N})[C]$ to the pca which is often called {\em Kleene's second model\/} ${\cal K}_2$: the underlying set is the set $\mathbb{N}^\mathbb{N}$ of all functions $\mathbb{N}\to \mathbb{N}$, where for $\alpha ,\beta\in \mathbb{N}^\mathbb{N}$, the application $\alpha\beta$ is defined if and only if for all $n\in \mathbb{N}$ there is a $k$ such that $\alpha ([n,\beta (0),\ldots ,\beta (k-1)])>0$, and if this is the case then
$$(\alpha\beta )(n)\, =\,\alpha ([n,\beta (0),\ldots ,\beta (k-1)])-1$$
for the least such $k$.

For the envisaged comparison, it is useful to first study the topological aspects of the pcas we have seen so far.

The pca ${\cal K}_2$ carries a natural topology, the {\em Baire space\/} topology, which is the product topology on the countable product of copies of $\mathbb{N}$ with the discrete topology; basic open sets are of the form $U_{\sigma}$ for a finite sequence $\sigma =(\sigma _0,\ldots ,\sigma _{n-1})$ of natural numbers, where
$$U_{\sigma}\, =\,\{\alpha\in \mathbb{N}^\mathbb{N}\, |\,\alpha (0)=\sigma _0,\ldots ,\alpha (n-1)=\sigma _{n-1}\}$$
We state the following fact, which is well-known, without proof.
\begin{proposition}\label{BaireK2}The Baire space topology is both conrep and repcon for ${\cal K}_2$.\end{proposition}

\noindent As observed by Andrej Bauer in \cite{Bauer}, there are applicative morphisms between ${\cal P}(\mathbb{N})$ and ${\cal K}_2$ in both directions. We have an applicative morphism $\iota :{\cal K}_2\to {\cal P}(\mathbb{N})$ (which is actually a single-valued relation) which sends $\alpha\in \mathbb{N}^\mathbb{N}$ to the set of its coded finite initial segments:
$$\iota (\alpha )\, =\,\{\{ [\alpha (0),\ldots ,\alpha (n-1)]\, |\, n\geq 0\}\}$$
In the other direction we have an applicative morphism $\delta :{\cal P}(\mathbb{N})\to {\cal K}_2$, where for $A\subseteq \mathbb{N}$, we have $\alpha\in\delta (A)$ if and only if $$A\, =\, \{ n\, |\,\exists i\in \mathbb{N}(\alpha (i)=n+1)\}$$
It is readily calculated (or see \cite{Bauer}) that for the compositions $\delta\iota$ and $\iota\delta$ we have: $\delta\iota$ is isomorphic to the identity on ${\cal K}_2$, and $\iota\delta\leq {\rm id}_{{\cal P}(\mathbb{N})}$. So the pair $(\iota ,\delta )$ forms an adjunction $\iota\dashv\delta$ in the 2-category {\bf PCA}, and $\delta$ is up to isomorphism a retraction on $\iota$.  

We can extend this adjunction to one involving ${\cal P}(\mathbb{N})[C]$ and a suitable extension of ${\cal K}_2$. Consider the following function $S:{\cal K}_2\to {\cal K}_2$:
$$S(\alpha )(n)\; =\;\left\{\begin{array}{rl}n+1 & \text{ if}\; n+1\not\in {\rm im}(\alpha ) \\ 0 & \text{otherwise}\end{array}\right.$$
Obviously, the map $S$ is not Baire continuous, so $\iota_S: {\cal K}_2 \to {\cal K}_2[S]$ is not an isomorphism.

Another easy observation is that ${\cal K}_2[S]$ is decidable: given $\alpha ,\beta\in \mathbb{N}^\mathbb{N}$, let
$$F(\alpha ,\beta )(n)\, =\,\left\{\begin{array}{rl} 1 & \text{if }\alpha (n)\neq\beta (n) \\ 0 & \text{otherwise}\end{array}\right.$$
Then $F$ is representable in ${\cal K}_2$, hence also in ${\cal K}_2[S]$; now $\alpha =\beta$ holds iff $1\not\in {\rm im}(F(\alpha ,\beta ))$, iff $S(F(\alpha ,\beta ))(0)=1$.

Let us now look at the adjoint pair $\iota\dashv \delta :{\cal K}_2\to {\cal P}(\mathbb{N})$. We also have $\iota _C:{\cal P}(\mathbb{N})\to {\cal P}(\mathbb{N})[C]$, and $\iota _S:{\cal K}_2\to {\cal K}_2[S]$.
\begin{lemma}\label{adjunctionextension} ~~~~~~~~~~~~~~~~~~~~~~~\begin{itemize}
		\item[a)] The complement function $C$ is representable in ${\cal K}_2[S]$ with respect to $\iota _S{\circ}\delta$.
		\item[b)] The function $S$ is representable in ${\cal P}(\mathbb{N})[C]$ with respect to $\iota _C{\circ}\iota$.\end{itemize}\end{lemma} 
\begin{proof} 
	a) We have $\alpha \in \iota _S{\circ}\delta (A)$ if and only if $\{ n\, |\, n+1\in {\rm im}(\alpha )\} =A$. So we need to find  a representable map on ${\cal K}_2[S]$ which sends each such $\alpha$ to some $\beta$ for which $\{ n\, |\, n+1\in {\rm im}(\beta )\} =\mathbb{N}-A$. But the map $S$ does precisely that, so we are done.
	
	\noindent b) We need to exhibit a representable operation in ${\cal P}(\mathbb{N})[C]$ which sends the set of initial segments of $\alpha$ to the set of initial segments of $S(\alpha )$. We use Scott-continuous operations and the complement function $C$. Continuously we get from $\{ [\alpha (0),\ldots ,\alpha (n-1)]\, |\, n\geq 0\}$ the set
	$$A \; =\; \{ n+1\, |\, n+1\in {\rm im}(\alpha )\} $$
	and from $A$, using $C$ and a continuous operation,
	$$B\; =\; C(A\cup\{ 0\})\; =\;\{ n+1\, |\, n+1\not\in {\rm im}(\alpha )\} $$
	From $A$ and $B$ we get
	$$D\; =\; \{ [0,n+1]\, |\, n+1\in A\}\cup\{ [1,n+1]\, |\, n+1\in B\}$$
	Now if $E$ is the set of all pairs $(\sigma ,a)$ satisfying:\begin{itemize}\item
		$\sigma$ is a coded sequence $[\sigma _0,\ldots ,\sigma _{n-1}]$ and for all $i<n$ there is a $j$ with $[j,i+1]\in e_a$, and whenever $[0,i+1]\in e_a$ then $\sigma _i=0$, and otherwise $\sigma _i=i+1$\end{itemize}
	then $E{\circ}D$ is the desired set of coded initial segments of $S(\alpha )$.
\end{proof}

\noindent It follows from lemma~\ref{adjunctionextension} that we have a commutative diagram
$$\xymatrix{ {{\cal K}_2}\ar[d]_{\iota _S}\ar[r]<-.6ex>_{\iota} & {{\cal P}(\mathbb{N})}\ar[l]<-.1ex>_{\delta}\ar[d]^{\iota _C} \\ 
	{{\cal K}_2[S]}\ar[r]<-.6ex>_{\iota _{\ast}} & {{\cal P}(\mathbb{N})[C]}\ar[l]<-.1ex>_{\delta _{\ast}} }$$ 
where ${\delta }_{\ast}$ is the unique factorisation of $\iota _S{\circ}\delta$ through ${\cal P}(\mathbb{N})[C]$ and $\iota _{\ast}$ the unique factorisation of $\iota _C{\circ}\iota$ through ${\cal K}_2[S]$.

It is a general feature of maps of the form $\iota _f:\mathcal{A}\to \mathcal{A}[f]$ that post-composition with $\iota _f$ reflects the preorder on pca morphisms: if $\gamma ,\delta :\mathcal{A}[f]\to \mathcal{B}$ satisfy $\gamma {\circ}\iota _f\leq\delta {\circ}\iota _f$, then $\gamma \leq \delta$. This is because $\iota_f$ is the identity relation. Therefore, in the diagram above we can conclude that $\delta _{\ast}{\circ}\iota _{\ast}\simeq {\rm id}_{{\cal K}_2[S]}$ and $\iota _{\ast}{\circ}\delta _{\ast}\leq {\rm id}_{{\cal P}(\mathbb{N})[C]}$.

\begin{lemma}\label{pullbacklemma} 
	The diagram
	$$\xymatrix{ {{\cal K}_2}\ar[r]^{\iota}\ar[d]_{\iota _S} & {{\cal P}(\mathbb{N})}\ar[d]^{\iota _C} \\
		{{\cal K}_2[S]}\ar[r]_{\iota _{\ast}} & {{\cal P}(\mathbb{N})[C]} } $$
	is a pullback diagram in {\bf PCA}.
\end{lemma}

\begin{proof} 
	Given a pair $(\gamma :\mathcal{A} \to {\cal P}(\mathbb{N}), \zeta :\mathcal{A} \to {\cal K}_2[S])$ such that $\iota _C{\circ}\gamma =\iota _{\ast}{\circ}\zeta$, we have $\delta {\circ}\gamma :\mathcal{A} \to {\cal K}_2$. We have $\iota {\circ}(\delta {\circ}\gamma )=(\iota {\circ}\delta ){\circ}\gamma\simeq\gamma$ and $\iota _S{\circ}(\delta {\circ}\gamma )\simeq (\iota _S{\circ}\delta ){\circ}\gamma\simeq \delta _{\ast}{\circ}\iota _C{\circ}\gamma =\delta _{\ast}{\circ}\iota _{\ast}{\circ}\zeta =\zeta$; by inserting some realizers we can obtain actual equality. Uniqueness of the factorization follows since $\iota _S$ is mono.
\end{proof}

\noindent Lemma~\ref{pullbacklemma} means that the function $S$ is, relative to the adjoint retraction $(\iota\dashv\delta :{\cal K}_2\to {\cal P}(\mathbb{N}))$, the restriction to ${\cal K}_2$ of the complement function.

We have also seen that ${\cal K}_2[S]$ is decidable. So, if for some pca $\mathcal{A}$ we define $Eq:\mathcal{A} \to \mathcal{A}$ to be the function which decides equality:
$$Eq ([a,b])\; =\; \left\{\begin{array}{ll} {\sf T}_{\mathcal{A}} & \text{if } \alpha = \beta \\ {\sf F}_{\mathcal{A}} & \text{otherwise}\end{array}\right.$$
then $Eq$ in $\mathcal{K}_2$ is representable in ${\cal K}_2[S]$.

\subsection{A variation on binary maps}\label{variationsection}
We can define a pca on $2^{\omega}$ to which the Cantor topology on $\mathcal{P}(\mathbb{N})$ is both conrep and repcon. We use the usual bijection between $2^{\omega}$ and $\mathcal{P}(\mathbb{N})$, where $\alpha \in 2^{\omega}$ is related to $A \subseteq \mathbb{N}$ if for all $n \in \mathbb{N}: \alpha(n) = 1 \Leftrightarrow n \in A$. We will often move from one description to the other. The definition of the application for $2^{\omega}$ goes in a similar fashion as $\mathcal{K}_2$: for $\alpha, \beta \in 2^{\omega}$:
\begin{itemize}
\item[] $\alpha \cdot \beta \downarrow \Leftrightarrow \forall n \exists k(\alpha([n,\beta(0),...,\beta(k-1)])=1)$
\end{itemize}
If $\alpha \cdot \beta \downarrow$, then $\alpha \cdot \beta(n) = \alpha([n,\beta(0),...,\beta(k)])$ for $k$ such that

\begin{itemize}
	\item[] $\forall i < k: \alpha([n,\beta(0),...,\beta(i-1)]) = 0$
	\item[] $\alpha([n,\beta(0),...,\beta(k-1)]) = 1$
\end{itemize}

\noindent Just like before, we can define an equality map $Eq$ over $2^{\omega}$. However, simply adding equality is not always enough to represent the maps we want. An alternative can be \textit{countable equality}, a map $Eq_{\infty}$ that interprets a set as a countable sequence of sets and checks equality pairwise:
\[
Eq_{\infty}(A,B) := \{n | \forall m: [ n,m ] \in A \Leftrightarrow [ n,m ] \in B \}
\]

\begin{theorem}\label{isoPN2o}
	The identity map from $\mathcal{P}(\mathbb{N})[C]$ to $2^{\omega}[Eq_{\infty}]$ is an isomorphism.
\end{theorem}

\begin{lemma}\label{isoPN2oright}
	$id: \mathcal{P}(\mathbb{N}) \rightarrow 2^{\omega}[Eq_{\infty}]$ is a decidable applicative morphism.
\end{lemma}

\begin{proof}
	We write the elements of $2^{\omega}$ as subsets of $\mathbb{N}$, just like we do for $\mathcal{P}(\mathbb{N})$. So we need to show that application in $\mathcal{P}(\mathbb{N})$, defined as $\circ: (A,B) \mapsto \{m | \exists n: [ n,m ] \in A, e_n \subset B\}$, is representable in $2^{\omega}$. 
	
	We know that any Cantor continuous map is representable in $2^{\omega}$. Let $I: 2^{\omega} \rightarrow 2^{\omega}$ be the map given by $I(A) = \{[ n,m ] | \exists m: e_m \subset A\}$. This map is Cantor continuous, hence representable. We let $II: 2^{\omega} \times 2^{\omega} \rightarrow 2^{\omega}$ be defined as $II(A,B) = A \cap I(B)$, which is also representable since taking intersection is Cantor continuous. Now let $\pi: 2^{\omega} \rightarrow 2^{\omega}$ be the projection $\pi(A) = \{m | \exists n: [ n,m ] \in A\}$. 
	\[
	\forall A,B: \pi(II(A,B)) = \pi(\{[ n,m ] \in A | e_n \subset B \}) = A \circ B
	\]
	So $\pi \circ II$ is the application from $\mathcal{P}(\mathbb{N})$. Note that $Eq_{\infty}(\emptyset,A) = C(\pi(A))$, where $C$ is the complement representable in $2^{\omega}$. So $\circ$ is representable.
\end{proof}

\noindent Since $C$ is representable in $2^{\omega}$, we know that $id: \mathcal{P}(\mathbb{N})[C] \rightarrow 2^{\omega}[Eq_{\infty}]$ is a decidable applicative morphism (by theorem \ref{extensiontheorem}). Note that since the Cantor topology is repcon for $2^{\omega}$, and conrep for $\mathcal{P}(\mathbb{N})[C]$, we immediately know that $id: 2^{\omega} \rightarrow \mathcal{P}(\mathbb{N})[C]$ is a decidable applicative morphism.

\begin{lemma}\label{isoPN2oleft}
	The map $id: 2^{\omega}[Eq_{\infty}] \rightarrow \mathcal{P}(\mathbb{N})[C]$ is a decidable applicative morphism.
\end{lemma}

\begin{proof}
	We show that $Eq_{\infty}$, as defined in $2^{\omega}$, can be represented in $\mathcal{P}(\mathbb{N})[C]$. Note that $(A \cap C(B)) \cup (C(A) \cap B)$ is the set consisting of those elements on which $A$ and $B$ differ. Taking the projection map $\pi(A) := \{n | \exists m: [ n,m ] \in A\}$, representable in $\mathcal{P}(\mathbb{N})$, we see that $\pi((A \cap C(B)) \cup (C(A) \cap B))$ consists of those $n$ such that there is an $m$ for which $A$ and $B$ differ on $[ n,m ]$. So $n$ is included if and only if $n \in Eq_{\infty}(A,B)$. Hence $Eq_{\infty} = C(\pi((A \cap C(B)) \cup (C(A) \cap B)))$.
\end{proof}

\noindent Hence we have established the identity map as an applicative morphism between $\mathcal{P}(\mathbb{N})[C]$ and $2^{\omega}[Eq_{\infty}]$ in both directions, so they must be isomorphisms and the two pcas represent the same maps. 

We define a map $\gamma: \mathcal{P}(\mathbb{N}) \to \mathcal{P}(2^{\omega})$ as follows. For $A \subset \mathbb{N}$ and $\alpha \in 2^{\omega}$ we have;
$$\alpha \in \gamma(A) \Leftrightarrow \forall n \in \mathbb{N}:(n \in A \Leftrightarrow \exists m \in \mathbb{N}: \alpha([n,m]) = 1)$$

\begin{proposition}
	The map $\gamma$ is a decidable applicative morphism $\mathcal{P}(\mathbb{N}) \to 2^{\omega}$.
\end{proposition}

\begin{proof}
	Within $\mathcal{P}(\mathbb{N})$ we define $f: \mathcal{P}(\mathbb{N}) \rightarrow \mathcal{P}(\mathbb{N})$ as $f([A,B]) = A \circ B$. Taking the pairing $[A,B]$ to be defined as $[A,B] := (2A) \cup (2B+1)$, we have that $$f(A) = \{m | \exists n: 2[n,m] \in A \wedge \forall k \in e_n: (2k+1) \in A\}$$
	We prove that $f$ is representable with respect to $\gamma$. For $A \subset N$ and $m \in \mathbb{N}$, take $A^{<m} := \{n \in A | n < m\}$. Let $r \in 2^{\omega}$ be such that for $\alpha \in \gamma$, $r \alpha([n.m]) = 1$ if $n \in f(A^{<m})$. Note that such an $r$ exists, since it only looks at $\alpha$ up to its $m$-th elements and then makes a decision. Now see that $\bigcup_m f(A^{<m}) = A$ and hence for all $n$:
	$$\exists m: r \alpha([n,m]) = 1 \Leftrightarrow \exists m: n \in f(A^{<m}) \Leftrightarrow n \in f(A)$$
	So $r \alpha \in \gamma(f(A))$. Decidability follows from the fact that we can have a $d \in 2^{\omega}$ such that $d \alpha(0) = i$ if there is an $m$ such that $\alpha([i,m]) = 1$ and for all $j \in 2, k \in \mathbb{N}$ such that $[j,k] < [i,m]$ we have $\alpha([i,m]) = 0$. This will function as a representation of decidability of $\gamma$.
\end{proof}

We look at the relation between $\mathcal{K}_2$ and $2^{\omega}$.
Let $\varepsilon: \mathcal{K}_2 \rightarrow 2^{\omega}$ be the map given by
$$\varepsilon(\alpha)([ n,m ]) := 
\begin{cases} 	1   & \mbox{if } \alpha(n) = m \\ 
0   	& \mbox{else }
\end{cases}$$

Note that applications of both $\mathcal{K}_2$ and $2^{\omega}$ look at begin sections of their input to make a decision on their output. Now, any begin section of $\alpha \in \mathcal{K}_2$ has a maximum of its elements, hence all the information of that begin section is stored in a begin section of $\epsilon(\alpha)$. More specifically:

\begin{proposition}\label{K2to2o}
	The map $\varepsilon$ is a decidable applicative morphism $\mathcal{K}_2 \to 2^{\omega}$.
\end{proposition}

\begin{proof}
	For $\alpha, \beta \in \mathcal{K}_2$, let $\gamma \in \mathcal{K}_2$ be their pairing such that $\gamma(2n) := \alpha(n)$ and $\gamma(2n+1) := \beta(n)$. Take $f \in \varepsilon(\gamma)$. We look at application in $\mathcal{K}_2$:
	\begin{itemize}
		\item[] $\alpha \cdot \beta(n) = m \Leftrightarrow$
		\item[] $\exists k,u_0,u_1,...,u_{k-1}: (\forall l<k: \beta(l)=u_l \wedge \alpha([n,u_0,u_1,...,u_{l-1}]) = 0)$
		\item[] \quad \quad \quad $\wedge \alpha([n,u_0,u_1,...,u_{k-1}]) = m+1 \Leftrightarrow$
		\item[] $\exists k,u_0,...,u_{k-1}: (\forall l<k: f([2l+1,u_l]) = 1 \wedge f([2[n,u_0,u_1,...,u_{l-1}],0]) = 1$
		\item[] \quad \quad \quad $\wedge f([2[n,u_0,u_1,...,u_{k-1}],m+1]) = 1$
	\end{itemize}
	
	So to represent a map from $f$ to $g \in \varepsilon(\alpha \cdot \beta)$ we design $r \in 2^{\omega}$ such that: $\forall n,m,s,k,u_0,u_1,...,u_{k-1}$ we take $\sigma \in 2^{\omega}$ the unique code of the sequence of length $[2[n,u_0,...u_{k-1}],s+1]+2$ such that: $\sigma(0) = [n,m]$, $\forall l<k: \sigma([2l,u_l]) = 1$, $\sigma([2[n,u_0,...,u_{l-1}],0]) = 1$ and $\sigma([2[n,u_0,...,u_{l-1}],s+1]) = 1$, then $r(\sigma) = 1$ and for all $i: r(\sigma \ast [i]) = 1 \Leftrightarrow s=m$. Take all other values of $r$ to be $0$. Then within $2^{\omega}$, $rf \in \varepsilon$. So application of $\mathcal{K}_2$ is representable with respect to $\varepsilon$. Decidability is easy to check.
\end{proof}

\begin{lemma}\label{gammafactor}
	The following diagram commutes:
\[
\xymatrix{
	\mathcal{P}(\mathbb{N}) \ar[r]_{\gamma} \ar[d]_{\delta} & 2^{\omega} \\
	\mathcal{K}_2 \ar[ur]_{\varepsilon} & 
}
\]
\end{lemma}

\begin{proof}
	For $A \in \mathcal{P}(\mathbb{N})$:
	\[
	\varepsilon(\delta(A)) = \bigcup \{\varepsilon(\alpha) | \text{im}(\alpha) - \{0\} = (A+1)\} =\]
	\[ 
	\{\beta | \forall n, \exists ! m, (\beta([ n,m ]) = 1) \cap \forall m (m \in A \Leftrightarrow \exists n, \beta([ n+1,m ]) = 1)\}
	\]
	Using transformation $t(\beta)([ n,m ]) = \beta([ m+1,n ])$, we have $t(\varepsilon(\delta(A))) = \{\beta | \forall m > 0, (\exists ! n: \beta([ n,m ]) = 1) \wedge \forall n(n \in A \Leftrightarrow \exists m, \beta([ n,m ]) = 1)\}$ which is a non-empty subset of $\gamma(A)$. So $\varepsilon \circ \delta \prec \gamma$, realized by $t$.
	
	Take $p: \mathbb{N} \rightarrow \mathbb{N} \times \mathbb{N}$ the bijective inverse of $[ , ]$. Let $t: 2^{\omega} \rightarrow 2^{\omega}$ be the representable map such that $t(\alpha)([ n,m ]) = 1 \Leftrightarrow (\alpha(p(n)) = 1 \wedge p_0(n)+1 = m) \vee (\alpha(p(n)) = 0)$. Take $\alpha \in \gamma(A)$. For a natural number $n$, either $\alpha(p(n)) = 1$ or $\alpha(p(n)) = 0$ and there is exactly one $m$ equal to $p_0(n)+1$. So $\forall n, \exists ! m, t(\alpha)([ n,m ]) = 1$. Secondly,  $\alpha \in \gamma(A)$ means that for any $n$, $n \in A$ is true if and only if there is an $m$ such that $\alpha([ n,m ]) = 1$. $\alpha([ n,m ]) = 1$ means $t(\alpha)([ [ n,m ], m+1 ]) = 1$, so $\exists k, t(\alpha)([ k,m+1 ]) = 1$. If on the contrary, for all $m$: $\alpha([ n,m ]) = 0$ we get that for all $k, t(\alpha)(k,m+1) = 0$ (since if $t(\alpha)(k,m+1) = 1$, $p_0(k) = m$ and $\alpha([ m,p_0(k) ]) = \alpha(p(k)) = 1$ which is against the assumption. So $\forall m (m \in A \Leftrightarrow \exists n, \alpha([ n,m+1 ]) = 1)$. We can conclude that $t(\alpha) \in \varepsilon \circ \delta$. So $\gamma \prec \varepsilon \circ \delta$
\end{proof}

\noindent A decidable applicative morphism in the other direction can be given by a map $\zeta: 2^{\omega} \rightarrow \mathcal{K}_2$ simply using the inclusion $2 \subset \mathbb{N}$.

\begin{proposition}\label{K2equiv}
	The applicative morphisms $\varepsilon$ and $\zeta$ give an equivalence between $2^\omega$ and $\mathcal{K}_2$.
\end{proposition}

\begin{proof}
	For $\alpha \in \mathcal{K}_2$ we have $\zeta(\varepsilon(\alpha))([n,m]) = 0 \Leftrightarrow \alpha(n) \neq m$,  $\zeta(\varepsilon(\alpha))([n,m]) = 1 \Leftrightarrow \alpha(n) = m$ and  $\zeta(\varepsilon((\alpha))([n,m]) < 2$. So it is easy to see that the maps $\alpha \mapsto \zeta(\varepsilon(\alpha))$ and its inverse (which is a partial map) are continuous in the Baire topology. This topology is conrep in $\mathcal{K}_2$, hence $\zeta \circ \varepsilon \sim {\sf id}_{\mathcal{K}_2}$. 
	
	Since the map $\varepsilon \circ \zeta$ has the same properties with respect to the Cantor topology, we can use the same argument to conclude that $\varepsilon \circ \zeta \sim {\sf id}_{2^{\omega}}$.
\end{proof}

\begin{proposition}\label{K2noniso}
	There cannot be an isomorphism between $2^\omega$ and $\mathcal{K}_2$.
\end{proposition}

\begin{proof}
	Assume such an isomorphism exists. It must be given by applicative morphisms both ways such that their compositions are equal to the identity. So it must be given by a bijective map $f: 2^\omega \to \mathcal{K}_2$ such that $f$ and $f^{-1}$ are applicative morphisms. We will prove that this $f$ is continuous.
	
	Consider the following set, $V = \mathcal{K}_2-\{{\sf T}_{\mathcal{K}_2}\}$ and note that it can be written as $\bigcup_{\neg \sigma \sqsubseteq {\sf T}_{\mathcal{K}_2}} U_{\sigma}$ and hence it is open in the Baire topology. Take $t$ the single element such that $f(t) = {\sf T}_{\mathcal{K}_2}$, then $W := \gamma^{-1}(V) = \gamma^{-1}(\mathcal{K}_2-\{{\sf T}_{\mathcal{K}_2}\}) = 2^{\omega} - \gamma^{-1}({\sf T}_{\mathcal{K}_2}) = 2^{\omega} - \{t\} = \bigcup_n U_{n \mapsto (1-t(n))}$ is open in the Cantor topology.
	
	Let $U_{\sigma}$ be a standard open in the Baire topology. Take $g_{\sigma}: \mathcal{K}_2 \rightarrow \mathcal{K}_2$ to be the map:
	$$g_{\sigma}(\alpha)(n) := \begin{cases} 
	{\sf F}_{\mathcal{K}_2} & \text{if } \alpha \in U_{\sigma}\\
	{\sf T}_{\mathcal{K}_2} & \text{otherwise}
	\end{cases}$$
	This function is continuous hence representable in $\mathcal{K}_2$ and we have $g^{-1}(V) = U_{\sigma}$. Because $f$ gives an isomorphism, $f^{-1}$ must form an applicative morphism, so there must be a representable (hence continuous) map $h: 2^{\omega} \rightarrow 2^{\omega}$ such that $f \circ h = g_{\sigma} \circ f$. So $f^{-1}(U_{\sigma}) = f^{-1}(g^{-1}(V)) = h^{-1}(f^{-1}(V)) = h^{-1}(W)$ which is open since $W$ is open and $h$ is continuous. So the inverse image through $f$ of any basic open is open, hence $f$ is continuous. 
	
	Now the Cantor topology is compact, so by $f$ we must conclude that the Baire topology is compact, which is not the case. We have a contradiction. So we cannot have an isomorphism.
\end{proof}

We see that $2^{\omega}$ and $\mathcal{K}_2$ give an example of two pcas that are equivalent but not isomorphic.

\section{Independence results}\label{indepenSection}
We have seen that we get decidability as a side-effect of adding the complement function to $\mathcal{P}(\mathbb{N})$. So in terms of the functions they represent, $\mathcal{P}(\mathbb{N})[C]$ is at least as powerful as $\mathcal{P}(\mathbb{N})[Eq]$. We can ask ourselves about the extent of this difference. The following result can be used to investigate the limits of what $\mathcal{P}(\mathbb{N})[Eq]$ and other similarly defined pcas can represent.

\begin{proposition}\label{countableRep}
	Given a pca $\mathcal{A}$ and a partial map $F: \mathcal{A} \rightarrow \mathcal{A}$ whose image is countable. Then for any partial map $f: \mathcal{A} \rightarrow \mathcal{A}$ representable in $\mathcal{A}[F]$, there is a countable partition $\{V_i\}_{i \in \mathbb{N}}$ of ${\rm dom}(f)$ such that for all $i:$ $f|_{V_i}$ is representable in $\mathcal{A}$.
\end{proposition}

\begin{proof}
	Let the element $a \in \mathcal{A}$ represent $f$. So for every $b \in \text{dom}(f)$, there is a sequence $c_0,\ldots ,c_{n-1}$ in $\mathcal{A}$ such that
	$$\begin{array}{rcl}
	a [b ,F(c_0),\ldots ,F(c_{i_1})] & = & [{\sf F},c_i]\; \text{for } i<n \\
	a [b ,F(c_0),\ldots ,F(c_{n-1})]) & = & [{\sf T}, f(b)] \end{array}$$
	Call the sequence $(F(c_0),\ldots ,F(c_{n-1}))$ a {\em computation sequence\/} for $b$. Now it is clear that if $V_{(F(c_0),\ldots ,F(c_{n-1}))}$ is the set of all $b$ with $(F(c_0),\ldots ,F(c_{n-1}))$ as computation sequence, then $f|_{V_{(F(c_0),\ldots ,F(c_{n-1}))}}$ is representable in $\mathcal{A}$. Now there are, by assumption on $F$, only countably many computation sequences, so the sets $V_{(F(c_0),\ldots ,F(c_{n-1}))}$ form a countable partition on the domain of $f$.
\end{proof}

\noindent Since $Eq$ is a function that only gives two values, we can use this result to say something about the representable maps in $\mathcal{A}[Eq]$ for certain pcas $\mathcal{A}$.

\begin{theorem}\label{Cnotrep}
	~\begin{itemize}
		\item[a)] The set $\mathcal{P}(\mathbb{N})$ cannot be written as a countable union $\bigcup_{i\in \mathbb{N}}V_i$ such that the complement function $C$ is Scott continuous on each $V_i$.
		\item[b)] The function $C$ is, relative to $\mathcal{P}(\mathbb{N})$, not computable in any function $F:\mathcal{P}(\mathbb{N}) \to \mathcal{P}(\mathbb{N})$ with countable image.
	\end{itemize}
\end{theorem}

\begin{proof}
	Suppose there is a partition $\{V_n\}_{n \in \mathbb{N}}$ of $\mathcal{P}(\mathbb{N})$ such that $C|_{V_n}$ is representable in $\mathcal{P}(\mathbb{N})$, hence continuous in the Scott topology.
	
	We will create a sequence of pairs of finite sets $\{(p_i,q_i)\}_{i \in \mathbb{N}}$ such that for each $i$: 
	\begin{enumerate}
		\item $p_i \cap q_i = \emptyset$
		\item $p_i \subset p_{i+1}$ and $q_i \subset q_{i+1}$
		\item $U_{p_i}^{q_i} \cap V_i = \emptyset$
	\end{enumerate}
	Let $p_0 = \emptyset = q_0$. Given two finite sets $(p_i,q_i)$ such that $p_i \cap q_i = \emptyset$. We construct $(p_{i+1},q_{i+1})$ in three cases. Let $n = \max(p_i \cup q_i)+1$. 
	\medskip
	
	\noindent {\sc case} 1: If $U_{p_i}^{q_i} \cap V_i = \emptyset$, we just take $p_{i+1} = p_i$ and $q_{i+1} = q_i$.
	\medskip
	
	\noindent{\sc case} 2: If $U_{p_i}^{q_i} \cap V_i \cap U^{\{n\}} \neq \emptyset$, take $A \in U_{p_i}^{q_i} \cap V_i \cap U_{n}$. Since $C|_{V_i}$ is Scott continuous, there is a Scott open $W$ such that $C^{-1}(U_{n}) \cap V_i = W \cap V_i$. Since $A \in U^{\{n\}} = C^{-1}(U_{n})$ and $A \in V_i$ we have $A \in W$. So there is a finite set $r$ such that $A \in U_r \subseteq W$. Since $A \in U_{p_i}^{q_i}$, $r \cap q_i = \emptyset$. We take $p_{i+1} = p_i \cup r \cup \{n\}$ and $q_{i+1} = q_i$. Then conditions 1 and 2 are satisfied. Take $A \in U_{p_{i+1}}^{q_{i+1}}$, then $r \subset A$ means $A \in W$ and $n \in A$ means $A \notin U^{\{n\}} = C^{-1}(U_{n})$, hence $A \notin V_i$. So $U_{p_{i+1}}^{q_{i+1}} \cap V_i = \emptyset$.
	\medskip
	
	\noindent{\sc case} 3: $U_{p_i}^{q_i} \cap V_i \cap U^{\{n\}} = \emptyset$. Take $p_{i+1} = p_i$ and $q_{i+1} = q_i \cup \{n\}$. Then conditions 1 and 2 are satisfied, and $U_{p_{i+1}}^{q_{i+1}} \cap V_i = U_{p_i}^{q_i} \cap U^{\{n\}} \cap V_i = \emptyset$.
	
	With such a sequence, $P := \bigcup_i p_i$ has the property that for all $n$, $P \cap q_n = \emptyset$. So for each $n$: $P \in U_{p_n}^{q_n}$ hence $P \notin V_n$. So $P$ is in not included in the partition of $\mathcal{P}(\mathbb{N})$. We have a contradiction and conclude that $C$ is not continuous over any countable partition of $\mathcal{P}(\mathbb{N})$, and by \ref{countableRep}, $C$ not representable in any $\mathcal{P}(\mathbb{N})[F]$ if $F$ has a countable image.
\end{proof}

\noindent We can conclude that the set of maps representable by $\mathcal{P}(\mathbb{N})[Eq]$ is a proper subset of the set of representable maps over $\mathcal{P}(\mathbb{N})[C]$.

We can use a similar argument when talking about $\mathcal{K}_2$. The following general result means that $S$ is never, relative to ${\cal K}_2$, computable in a function with countable image (such as $Eq$).
\begin{theorem}\label{Snotcomp}
	~\begin{itemize}
		\item[a)] The set $\mathbb{N}^\mathbb{N}$ cannot be written as a countable union $\bigcup_{i\in \mathbb{N}}V_i$ such that $S$ is Cantor continuous on each $V_i$.
		\item[b)] The function $S$ is, relative to ${\cal K}_2$, not computable in any function $F:\mathbb{N}^\mathbb{N}\to \mathbb{N}^\mathbb{N}$ with countable image.
	\end{itemize}
\end{theorem}

\begin{proof} 
	Suppose $(V_i)_{i\in \mathbb{N}}$ is a collection of subsets of $\mathbb{N}^\mathbb{N}$ such that $S$ is continuous on each $V_i$. We shall construct an $\alpha\in \mathbb{N}^\mathbb{N}$ such that $\alpha\not\in\bigcup_{i\in \mathbb{N}}V_i$.
	
	To this end we construct a sequence of pairs $(\sigma _i, \rho _i)_{i\in \mathbb{N}}$ with the following properties: $\sigma _i$ is a finite sequence of numbers, $\rho _i$ is a finite set of numbers, and the following hold:\begin{itemize}
		\item[i)] ${\rm im}(\sigma _i)\cap\rho _i=\emptyset$
		\item[ii)] $\sigma _i$ is an initial segment of $\sigma _{i+1}$; $\rho _i\subseteq\rho _{i+1}$
		\item[iii)] writing $U_{\sigma}^{\rho}=\{\alpha\in \mathbb{N}^\mathbb{N}\, |\, \sigma \text{ is an initial segment of } \alpha ,{\rm im}(\alpha )\cap\rho =\emptyset\}$, we have $U_{\sigma _{i+1}}^{\rho _{i+1}}\cap V_i=\emptyset$\end{itemize}
	Clearly, given such a sequence, there is an $\alpha\in\bigcap_{i\in \mathbb{N}}U_{\sigma _i}^{\rho _i}$, and this $\alpha$ cannot be in any $V_i$.
	
	Now for the construction: let $\sigma _0$ be the empty sequence; $\rho _0=\emptyset$.
	
	Suppose $(\sigma _i,\rho _i)$ have been constructed. Let $m$ be the first number such that $m+1\not\in {\rm im}(\sigma _i\cup\rho _i$. We consider the set
	$$Z_i\, =\, V_i\cap U_{\sigma _i}^{\rho _i}\cap\{\alpha\in \mathbb{N}^\mathbb{N}\, |\, m+1\not\in {\rm im}(\alpha )\}$$
	Note that $\{\alpha \in \mathbb{N}^\mathbb{N}\, |\, m+1\not\in {\rm im}(\alpha )\} = S^{-1}(\{\alpha\in \mathbb{N}^\mathbb{N}\, |\, \alpha (m)=m+1\} )$.
	Since $\{\alpha \in \mathbb{N}^\mathbb{N}\, |\, \alpha (m)=m+1\}$ is open in the Baire space topology and $S$ is continuous on $V_i$, we have an open set $W$ such that
	$$Z_i\, =\, V_i\cap U_{\sigma _i}^{\rho _i}\cap W$$
	We distinguish two cases:
	
	\noindent {\sc case A:} $Z_i \neq \emptyset$. There must be some extension $\tau$ of $\sigma _i$ such that $V_i\cap U_{\tau}\cap U_{\sigma _i}^{\rho _i}$ is a nonempty subset of $Z_i$. Let $\sigma _{i+1}$ be $\sigma _i{\ast}(m+1)$ ($m+1$ appended to $\tau$ as last element); let $\rho _{i+1}=\rho _i$. Then i) and ii) are satisfied; and if $\alpha \in U_{\sigma _{i+1}}^{\rho _{i+1}}$ then $\alpha\in U_{\tau}\cap U_{\sigma _i}^{\rho _i}$, so, since $\alpha \not\in Z_i$, we must have $\alpha\not \in V_i$. So iii) holds as well.
	\medskip
	
	\noindent {\sc case B:} $Z_i =\emptyset$. Then for all $\alpha\in V_i\cap U_{\sigma _i}^{\rho _i}$ we have $m+1\in {\rm im}(\alpha )$. Let $\sigma _{i+1}=\sigma _i$; $\rho _{i+1}=\rho _i\cup\{ m+1\}$. Again, i) and ii) are satisfied and for $\alpha\in U_{\sigma _{i+1}}^{\rho _{i+1}}$ we cannot have $\alpha\in V_i$. This finishes the construction of the sequence $(\sigma _i,\rho _i)$ and proves part a) of the theorem.
	\medskip
	
	Part b) is a consequence of part a) and proposition \ref{countableRep}.
\end{proof}

\noindent We see that the recursion theory of ${\cal K}_2$ is radically different from the ordinary case: the function $S$ is, for example, not computable in its own graph seen as a characteristic function of ordered pairs. A similar conclusion holds for $\mathcal{P}(\mathbb{N})$.

Given the equivalence between $2^{\omega}$ and $\mathcal{K}_2$ in \ref{K2equiv} and theorem \ref{Snotcomp}, the following theorem does not come as a surprise.

\begin{theorem}\label{counteqnotrep}
	~\begin{itemize}
		\item[a)] The set $2^{\omega}$ cannot be written as a countable union $\bigcup_{i\in \mathbb{N}}V_i$ such that $Eq_{\infty}$ is Cantor continuous on each $V_i$.
		\item[b)] The function $Eq_{\infty}$ is, relative to $2^{\omega}$, not computable in any function $F:2^{\omega} \to 2^{\omega}$ with countable image.
	\end{itemize}
\end{theorem}

\begin{proof}
	Some notation first, for $\sigma \in 2^*$, write\\ 
	$U_{\sigma} := \{\alpha \in 2^{\omega} | \sigma \text{ is an initial segment of } \alpha \}$ and $U_{n \mapsto t} := \{\alpha \in 2^{\omega} | \alpha(n) = t \}$. These are Cantor open sets. Take $P: 2^{\omega} \rightarrow 2^{\omega}$ the projection map where $P(\alpha)(n) = 1 \Leftrightarrow \exists m: \alpha([ n,m ]) = 1$. Note that there is a Cantor continuous maps $f$ such that $P = Eq_{\infty} \circ f$. So we will prove the theorem for $P$ instead of $Eq_{\infty}$.
	
	Assume there is a countable partition $\{V_i\}_{i \in \mathbb{N}}$ of $2^{\omega}$ such that for all $i$, $P|_{V_i}$ is Cantor continuous. We inductively define a sequence of compatible partial maps $\mathbb{N} \rightarrow 2$, beginning with the map $f_0$ that is nowhere defined. For any partial map $f$, we define
	\[
	W_f := \{\alpha \in 2^{\omega} | \forall n: f(n)\downarrow \Rightarrow \alpha(n) = f(n)\}
	\]
	So $W_{f_0} = 2^{\omega}$. In each step, from $f_i$ we construct an extension $f_{i+1}$ such that $\forall n: f_i(n)\downarrow \Rightarrow (f_{i+1}(n) = f_i(n))$ and $W_{f_i} \cap V_i = \emptyset$. If such a sequence exist, then there is an extension $f \in 2^{\omega}$ of all $f_i$ such that $f \notin V_i$ for all $i$. Which is impossible, since the $V_i$-s form a partition.
	
	Let $p_n: 2^{\omega} \rightarrow 2^{\omega}$ be the $n$-th projection: $p_n(\alpha)(m) = \alpha([ n,m ])$. During the construction of the $f_i$-s, we will also prove for each $f_i$ that there is a $\nu_i$ with the following property:
	\[
	\forall m \geq \nu_i: p_m(W_{f_i}) = 2^{\omega} (\Leftrightarrow \forall m \geq \nu_i, \forall k \in \mathbb{N}: f_i([ k,m ])\uparrow)
	\]
	In case of $f_0$ we have for all $m \in \mathbb{N}$: $p_m(W_{f_0}) = p_m(2^{\omega}) = 2^{\omega}$. So we can take $\nu_0 = 0$.
	
	Assume we have $f_i$ and $\nu_i$ such that $\forall m \geq \nu_i: p_m(W_{f_i}) = 2^{\omega}$. If $W_{f_i} \cap V_i = \emptyset$, we just take $f_{i+1} = f_i$ and $\nu_{i+1} = \nu_i$. Now assume $W_{f_i} \cap V_i \neq \emptyset$. Since $P|_{V_i}$ is Cantor continuous, there is a Cantor open $O$ such that $P^{-1}(U_{\nu_i \mapsto 1}) \cap V_i = O \cap V_i$. We distinguish two cases. 
	\medskip
	
	\noindent {\sc case} 1: $W_{f_i} \cap O \neq \emptyset$, take some $\alpha \in W_{f_i} \cap O$. Since $O$ is a Cantor open, $O = \bigcup_j U_{\sigma_j}$ for certain finite sequences $\sigma_j$. $\alpha \in O$ means there is a $j$ such that $\sigma_j$ is an initial segment of $\alpha$. Let the partial map $g: \mathbb{N} \rightarrow 2$ be the extension of both $f_i$ and $\sigma_j$ ($g(m) = \sigma_j(m)$ if $m < lh(\sigma_j)$, else $g(m) = f_i(m)$). Here, $lh$ gives the length of the sequence. So $\alpha \in W_g$. Since $p_{\nu_i}(W_{f_i}) = 2^{\omega}$ and $\sigma_j$ is finite, we can find an $m$ such that $g([ \nu_i,m ])\uparrow$. Let $f_{i+1}$ be the extension of $g$ defined on $[ n,m ]$ as $1$. So $W_{f_{i+1}} = W_{f_i} \cap U_{\sigma_j} \cap U_{[ \nu_i,m ] \rightarrow 1}$. For $\beta \in W_{f_{i+1}}$, we have $\beta \in O$ since $\beta \in U_{\sigma_j}$ and $\beta \in P^{-1}(U_{\nu_i \mapsto 0})$ since $\beta([ \nu_i,m ]) = 1$. So $\beta \notin V_i$. Secondly, since $\sigma_j$ is finite, there are only finitely many additions to $f_i$, so there must be an $n$ such that $\forall m \geq n: p_m(W_{f_{i+1}}) = 2^{\omega}$. Take $\nu_{i+1}$ such an $n$.
	\medskip
	
	\noindent {\sc case} 2: $W_{f_i} \cap O = \emptyset$. This means that for any $\alpha \in W_{f_i} \cap f_i$, we have $P(\alpha)(\nu_i) = 0$, so there must be an $m$ such that $\alpha([ \nu_i,m ]) = 1$. Let $f_{i+1}$ be the extension of $f_i$ where $f_{i+1}([ \nu_i,m ]) = 0$ for all $m$, and everywhere else $f_{i+1}(m) = f_i(m)$ (Note that $f_i$ was not yet defined on those $[ \nu_i,m ]$ because of the $\nu_i$ condition). So, for each $\beta \in W_{f_{i+1}}$ we have $P(\beta)(\nu_i) = 1$, hence $W_{f_{i+1}} \cap V_i = \emptyset$. Note that we can take $\nu_{i+1} = \nu_i+1$, since we only extended over the $[ \nu_i,m ]$-s.
	
	That finishes the construction of the $f_i$. Since they are compatible, there is an $f: \mathbb{N} \rightarrow 2$ extending all of them. For this $f$ we have for all $i \in \mathbb{N}$: $\{f\} \cap V_i \subseteq W_{f_i} \cap V_i = \emptyset$. So $f$ is not included in the partition. This is a contradiction. So $P$ is not continuous over a countable partition of $2^{\omega}$, so neither is $Eq_{\infty}$. Part b) follows from \ref{countableRep}.
\end{proof}

\noindent We have seen that adding equality to a pca adds only limited computational power. Could it be that the pca with $Eq$ added, could still be simulated in the old one? For instance in the case of $\mathcal{P}(\mathbb{N})$, $\mathcal{K}_2$ and $2^{\omega}$?

\begin{proposition}\label{countBasis}
	Let $\mathcal{A}$ be a decidable pca with uncountably many elements. Let $\mathcal{B}$ be a pca such that there is a countably based non-trivial topology which is repcon for $\mathcal{B}$. Then there is no decidable applicative morphism from $\mathcal{A}$ to $\mathcal{B}$.
\end{proposition}

\begin{proof}
	Let $\gamma: \mathcal{A} \to \mathcal{B}$ be a decidable applicative morphism, and let $\mathcal{T}$ be a non-trivial countably based topology which is repcon for $\mathcal{B}$. Let $\{U_i\}_{i \in \mathbb{N}}$ be a countable basis of $\mathcal{T}$ and take $\Gamma = \bigcup_{a \in \mathcal{A}} \gamma(a)$.
	
	Let $U$ be a non-trivial open in $\mathcal{T}$. Take $x \in U$ and $y \in \mathcal{B}-U$. By definition of the Booleans, there is a representable map sending ${\sf T}_{\mathcal{B}}$ to $x$ and ${\sf F}_{\mathcal{B}}$ to $y$. Since $\mathcal{T}$ is repcon for $\mathcal{B}$ there must be an open $V$ such that ${\sf T}_{\mathcal{B}} \in V$ and ${\sf F}_{\mathcal{B}} \notin V$.
	
	Take an element $x \in \mathcal{A}$ and denote $d_x: \mathcal{A} \rightarrow \mathcal{A}$ the map $d_x(a) := Eq(x,a)$ which is representable in $\mathcal{A}$ because of decidability. Since $\gamma$ is a decidable applicative morphism, we can construct a partial map $r_x: \mathcal{B} \rightarrow \mathcal{B}$ such that:
	\[
	\forall a \in \mathcal{A}, b \in \gamma(a): r_x(b) := \begin{cases} 	{\sf T}_{\mathcal{B}}   & \mbox{if } a=x \\
	{\sf F}_{\mathcal{B}} 	& \mbox{otherwise }
	\end{cases}
	\]
	Note that for $y \in \mathcal{A}$ with $y \neq x$ we have $r_x(\gamma(y)) = \{{\sf F}_{\mathcal{B}}\}$ so $\gamma(x) \cap \gamma(y) = \emptyset$. So $\{\gamma(y)\}_{y \in \mathcal{A}}$ forms a partition of $\Gamma$.
	
	Now, since $\mathcal{T}$ is repcon for $\mathcal{B}$, there must be an open $U$ such that $U \cap \text{dom}(r_x) = r_x^{-1}(V)$. We have $\Gamma \subset \text{dom}(r_x)$, so $U \cap \Gamma = r_x^{-1}(V) \cap \Gamma$. Take $v \in \gamma(x) \neq \emptyset$, then $v \in \Gamma$ and $r_x(v) = {\sf T}_{\mathcal{B}} \in V$. So $r_x^{-1}(V) \cap \Gamma \neq \emptyset$. 
	Now take $b \in r_x^{-1}(V) \cap \Gamma$. Since $b \in \Gamma$ there is an $a \in \mathcal{A}$ such that $b \in \gamma(a)$. Since $b \in r_x^{-1}(V)$ we have $r_x(b) \neq {\sf F}_{\mathcal{B}}$, so $a = x$. We get that $\emptyset \neq r_x^{-1}(V) \cap \Gamma \subset \gamma(x)$. Since $\mathcal{T}$ has a countable basis and $b \in U$ there must be an $i \in \mathbb{N}$ such that $b \in U_{i_x}$. We can conclude that $\emptyset \neq U_i \cap \Gamma \subset U \cap \Gamma = r_x^{-1}(V) \cap \Gamma \subset \gamma(x)$.
	
	So for any $x \in \mathcal{A}$ there is an $i \in \mathbb{N}$ such that $\emptyset \neq U_i \cap \Gamma \subset \gamma(x)$. Since for $x \neq y$ we have $\gamma(x) \cap \gamma(y) = \emptyset$, we must have a distinct $i \in \mathbb{N}$ for each $x \in \mathcal{A}$. But this is in contradiction with the fact that $\mathcal{A}$ is uncountable. We conclude $\gamma$ cannot exist.
\end{proof}

\noindent Now, since the Scott, Baire and Cantor topologies do have countable bases, we get the following direct consequence.

\begin{corollary}\label{eqismore}
	There are no decidable applicative morphisms from $\mathcal{P}(\mathbb{N})[Eq]$, $\mathcal{K}_2[Eq]$ and $2^{\omega}[Eq]$ into $\mathcal{P}(\mathbb{N})$, $\mathcal{K}_2$ and $2^{\omega}$.
\end{corollary}

\section{Recursive aspects}\label{recursivesection}
A \textit{sub-pca} $\mathcal{B}$ of a pca $\mathcal{A}$ is a pca defined on a subset of $\mathcal{A}$, inheriting the applicative structure of $\mathcal{A}$ and containing some choice of $k$ and $s$ functioning as the appropriate combinators for both $\mathcal{A}$ and $\mathcal{B}$.

Let $RE \subset \mathcal{P}(\mathbb{N})$ be the recursively enumerable subsets of $\mathbb{N}$. With the application from $\mathcal{P}(\mathbb{N})$, $RE$ forms a sub-pca of $\mathcal{P}(\mathbb{N})$ (see \cite{Bauer}). Note that for any $A \in RE$, we have that its complement $C(A)$ is in $RE$, precisely if it is recursive. So in $RE$, $C$ is a partial map defined on the subset $Rec \subset RE$ of recursive sets. By the same proof of decidability of $\mathcal{P}(\mathbb{N})[C]$ we can see that $Eq|_{Rec}$ is representable in $RE[C]$.

Now consider the following set: $Uni = \{[ n,m ] | \phi_n(m) \downarrow \}$, containing the pairs $n$ and $m$ such that the $n$-th Turing machine halts with input $m$. Since we have a universal Turing machine, this set is recursively enumerable. It also contains all $RE$ sets, meaning for any $A \in RE$, there is an $n$ such that $Uni_n := \{m \in \mathbb{N} | [ n,m ] \in Uni \} = A$. This allows us to enumerate all $RE$ sets and search through them.

\begin{lemma}\label{UniRec}
	$C$ is representable in $RE[Eq]$.
\end{lemma}

\begin{proof}
	We define $Eq$ as a map on a single argument $Eq'$, so $Eq'([A,B]) = Eq(A,B)$. For $A$ and $B$, we have that $B = C(A)$ if and only if $A \cup B = \mathbb{N}$ and $A \cap B = \emptyset$. To combine the two into a single check, we can write a function representable in $RE$ defined as $Ic(A,B) := [[A \cup B,A \cap B],[\mathbb{N},\emptyset]]$. Then $Eq'(Ic(A,B)) = {\sf T} \Leftrightarrow (B = C(A))$. Now we want an algorithm that checks this for all $Uni_n$, a $Z \in RE$ such that for all $n$ and $U_0 = \{0\},...,U_{n-1}=\{0\}$ we have $Z [A,U_0,...,U_{n-1}] = [{\sf F},Ic(A,Uni_n)]$ and $Z [A,U_0,...,U_{n-1},\{1\}] = [{\sf T},Uni_n]$. Then $ZA = Uni_n$ such that $Uni_n = C(A)$ (if it exists). But such a $Z \in RE$ can simply be given by 
	\begin{itemize}
	\item[] $Z := {<} x {>} {\sf if} \quad (lh(x)=0) \quad {\sf then} \quad Uni_0 \quad {\sf else}$ 
	\item[] $\quad \quad \quad \quad \quad {\sf if} \quad ({\sf lst}\text{ }x) \quad {\sf then} \quad [{\sf T},Ic({\sf fst} \text{ }x,Uni_{lh(x)})] \quad {\sf else} \quad [{\sf F},Uni_{lh(x)}]$
	\end{itemize}
	Where ${\sf fst}$ and ${\sf lst}$ respectively give the first and the last element of a sequence, and $lh$ gives the length of a sequence.
\end{proof}

\noindent Note that this algorithm does not halt if $A$ does not have a complement (is not in $Rec$), which is fine since $C$ is not defined there. Secondly, note that we cannot do this trick using $Eq|_{Rec}$, since by ranging over all $RE$ sets, we also need to check equality for non-$Rec$ sets. 

We take $R \subset 2^{\omega}$ to be the subset of recursive $0-1$ sequences. Similarly to the recursive sub-pca of $\mathcal{K}_2$ in \cite{Bauer}, we can see $R$ as a sub-pca of $2^{\omega}$. The usual bijection $b: 2^{\omega} \to \mathcal{P}(\mathbb{N})$ also gives a bijection $in = b|_{R}: R \to Rec$ where $Rec \subset RE$. So we have a way to relate $R$ to $RE$.

\begin{lemma}\label{RectoRE}
	The injective function $in: R \rightarrow RE$ forms a decidable applicative morphism from $R[Eq]$ to $RE[C]$
\end{lemma}

\begin{proof}
	Take $(-)': RE[C] \rightarrow RE[C]$ to be the representable map $X \mapsto (2X) \cup (2C(X))$. We want to represent application $\cdot_{R}$ of $R \subset 2^{\omega}$ in $RE[C]$, by factoring it through the map $(-)'$. So we want to represent a map which for each $A$ and $B$ sends $(A',B')$ to $A \cdot_{Rec} B$ in $RE \subset \mathcal{P}(\mathbb{N})$. We do that by taking $Z \in \mathcal{P}(\mathbb{N})$ to be the set containing elements of the form $[v,[n,m]]$ such that $\exists u_0,u_1,...,u_{l+1}$:
	\[
	e_m := \{2i | u_i = 1\} \cup \{2i+1 | u_i = 0\}
	\]
	\[
	e_v := \{2[n,u_0,...,u_k],2[n,u_0,...,u_{k+1}]\} \cup \{2[n,u_0,...,u_{l-1}] | l \leq k\}
	\]
	Note that because of the finiteness of the $e_m$ and $e_p$ and the computability of the enumeration of finite sets $e_{(-)}$, we can computably check whether $[v,[n,m]]$ has this property, hence $Z \in RE$. Now note that if $A \cdot_{R} B \downarrow$, $(A \cdot_{R} B)(n) = 1$ precisely when $n \in Z A' B'$. So in that case $Z A' B' = A \cdot_{R} B$. We can conclude that $in: R \rightarrow RE[C]$ is an applicative morphism realized by ${<} AB {>} Z[A,C(A)][B,C(B)]$. Since $Eq|_{Rec}$ is realized in $RE[C]$, we have that the equality $Eq$ in $R$ is representable with respect to $in$. So $in: R[Eq] \rightarrow RE[C]$ is a decidable applicative morphism.
\end{proof}

\noindent We conclude the following.

\begin{corollary}\label{RecDiagram}
	We have a system of decidable applicative morphisms:
	\[
	\xymatrix{
		RE[Eq|_{Rec}] \ar[r]^{id} &
		RE[C] \ar[r]^{id} &
		RE[Eq]\\
		& R[Eq] \ar[u]^{in} &
	}
	\]
\end{corollary}

\subsection*{Acknowledgements}{\small}
This material is based upon work supported by the Air Force Office of Scientific Research, Air Force Materiel Command, USAF under Award No.
FA9550-14-1-0096

\end{document}